\documentclass[12pt]{amsart}
\usepackage{amsmath,amssymb,amsthm,graphics,amscd,mathrsfs}
\usepackage{amscd, amsfonts, amssymb, amsthm}
\usepackage{fullpage, verbatim}
\usepackage[colorlinks, breaklinks, linkcolor=blue]{hyperref} 
\usepackage{breakurl}
\usepackage{array}
\usepackage{latexsym}
\usepackage{enumerate}

\newcommand\CA{{\mathscr A}} 
\newcommand\CB{{\mathscr B}}

\newcommand\CI{{\mathcal I}}  
\newcommand\CJ{{\mathscr J}}  
\newcommand\CK{{\mathscr K}} 

\newcommand\CM{{\mathcal M}}

\newcommand\FS{\mathscr S}

\newcommand\BBC{{\mathbb C}}

\newcommand\BBP{{\mathbb P}}

\newcommand\BBR{{\mathbb R}}
\newcommand\BBZ{{\mathbb Z}}

\newcommand\codim{\operatorname{codim}}

\newcommand\hgt{\operatorname{ht}}

\newcommand\id{{\operatorname{id}}}

\numberwithin{equation}{section}

\theoremstyle{plain}
\newtheorem{lemma}[equation]{Lemma}
\newtheorem{theorem}[equation]{Theorem}

\newtheorem{conjecture}[equation]{Conjecture}
\newtheorem{condition}[equation]{Condition}
\newtheorem{corollary}[equation]{Corollary}
\theoremstyle{definition}
\newtheorem{defn}[equation]{Definition}
\newtheorem{remark}[equation]{Remark}

\newtheorem{example}[equation]{Example}


\begin{document}

\title[The topology of arrangements of ideal type]
{The topology of arrangements of ideal type}

\author[N. Amend]{Nils Amend}
\address
{Institut f\"ur Algebra,~Zahlentheorie und Diskrete Mathematik,
Fakult\"at f\"ur Mathematik und Physik,
Gottfried Wilhelm Leibniz Universit\"at Hannover,
Welfengarten 1, D-30167 Hannover, Germany}
\email{amend@math.uni-hannover.de}

\author[G. R\"ohrle]{Gerhard R\"ohrle}
\address
{Fakult\"at f\"ur Mathematik,
Ruhr-Universit\"at Bochum,
D-44780 Bochum, Germany}
\email{gerhard.roehrle@rub.de}


\allowdisplaybreaks

\begin{abstract}
In 1962, Fadell and Neuwirth showed that the 
configuration space of the braid arrangement
is aspherical. Having generalized this to
many real reflection groups, Brieskorn conjectured this for 
all finite Coxeter groups. 
This in turn follows from Deligne's seminal work from 1972,
where he showed that the complexification of every real 
simplicial arrangement is a $K(\pi,1)$-arrangement.

In this paper we study
the $K(\pi,1)$-property for  
a certain class of subarrangements of Weyl arrangements,
the so called  
arrangements of ideal type $\CA_\CI$. These stem
from ideals $\CI$ in the set of positive 
roots of a reduced root system.
We show that the $K(\pi,1)$-property 
holds for all arrangements $\CA_\CI$ if the underlying Weyl group 
is classical and that it extends to most of the
$\CA_\CI$ if the underlying Weyl group is of exceptional type.
Conjecturally this holds for all $\CA_\CI$.
In general, the  $\CA_\CI$ are neither simplicial, 
nor is their complexification fiber-type.
\end{abstract}

\maketitle


\section{Introduction and Results}
By fundamental work of 
Fadell-Neuwirth \cite{fadellneuwirth}, Brieskorn \cite{brieskorn:tresses} 
and Deligne \cite{deligne},
all Coxeter arrangements are 
$K(\pi,1)$-arrangements, i.e.~the complements of 
their complexifications are aspherical spaces.

While Coxeter arrangements are well studied, 
their subarrangements 
are considerably less well understood. 
In this paper we study the topology of the complements of 
certain arrangements which are 
associated with ideals in the set of positive roots of 
a reduced root system, so called 
\emph{arrangements of ideal type}  $\CA_\CI$,
Definition \ref{def:idealtype},
cf.~\cite[\S 11]{sommerstymoczko}.
We show that a combinatorial property introduced in 
\cite[Cond.\ 1.10]{roehrle:ideal} 
combined with Terao's fibration theorem \cite{terao:modular}
gives an inductive method to show that a large class of
(the complexifications of) 
the arrangements of ideal type $\CA_\CI$
are indeed $K(\pi,1)$-arrangements. 
This inductive technique was used in 
\cite{roehrle:ideal} 
to show that many of the arrangements $\CA_\CI$
are inductively free.
In general a subarrangement of a Weyl arrangement need not be $K(\pi,1)$,
e.g.~see Example \ref{ex:nonkpione}.

Let $\Phi$ be an irreducible, reduced root system
and let $\Phi^+$ be the set of positive roots 
with respect to some set of simple roots $\Pi$.
An \emph{(upper) order ideal}, 
or simply \emph{ideal} for short, 
of  $\Phi^+$, is a subset $\CI$  of $\Phi^+$ 
satisfying the following condition: 
if $\alpha \in \CI$ and $\beta \in \Phi^+$ so that 
$\alpha + \beta \in \Phi^+$, then $\alpha + \beta \in \CI$.
Recall the standard partial ordering 
$\preceq$ on $\Phi$: $\alpha \preceq \beta$
provided $\beta - \alpha$ is a $\BBZ_{\ge0}$-linear combination 
of positive roots, or $\beta = \alpha$. Then $\CI$ is an ideal in $\Phi^+$
if and only if whenever 
$\alpha \in \CI$ and $\beta \in \Phi^+$ so that 
$\alpha \preceq \beta$, then $\beta \in \CI$.

Let $\beta$ be in $\Phi^+$. Then $\beta = \sum_{\alpha \in \Pi} c_\alpha \alpha$
for $c_\alpha \in \BBZ_{\ge0}$.
The \emph{height} of $\beta$ is defined to be $\hgt(\beta) = \sum_{\alpha \in \Pi} c_\alpha$.
Let $\CI \subseteq \Phi^+$ be an ideal and let 
\[
\CI^c := \Phi^+ \setminus \CI
\]
be its complement in 
$\Phi^+$. 

Following \cite[\S 11]{sommerstymoczko}, 
we associate with an ideal $\CI$ in $\Phi^+$ the arrangement 
consisting of all hyperplanes with respect to the roots in $\CI^c$.
Let $\CA(\Phi)$ be the \emph{Weyl arrangement} of $\Phi$,
i.e., $\CA(\Phi) = \{ H_\alpha \mid \alpha \in \Phi^+\}$,
where $H_\alpha$ is the hyperplane in the Euclidean space
$V = \BBR \otimes \BBZ \Phi$ orthogonal to the root $\alpha$.

\begin{defn}[{\cite[\S 11]{sommerstymoczko}}]
\label{def:idealtype}
Let $\CI \subseteq \Phi^+$ be an ideal.
The \emph{arrangement of ideal type} associated with 
$\CI$ is the subarrangement $\CA_\CI$
of $\CA(\Phi)$ defined by 
\[
\CA_\CI := \{ H_\alpha \mid \alpha \in \CI^c\}.
\]
\end{defn}

It was shown by Sommers and Tymoczko \cite[Thm.\ 11.1]{sommerstymoczko}
that each $\CA_\CI$ is free
in case the root system is classical or of type $G_2$.
The general case was settled
in a uniform manner for all types 
by Abe, Barakat, Cuntz, Hoge, and Terao
in \cite[Thm.\ 1.1]{abeetall:weyl}. 
The non-zero exponents 
are given by the dual of the height partition of the roots in $\CI^c$.

Note that the complement $\CI^c$ forms a lower ideal in $\Phi^+$.
Thus in particular, in type $A_n$ the arrangements of ideal type $\CA_\CI$
are graphic arrangements corresponding to chordal 
graphs on $n+1$ vertices. The freeness of the latter
is due to Stanley, \cite[Prop.~2.8]{stanley:super}.

In \cite[Cor.\ 5.15]{barakatcuntz:indfree}, 
Barakat and Cuntz showed that every Weyl arrangement $\CA(\Phi)$  
is \emph{inductively free}.
It was shown in 
\cite{roehrle:ideal} that 
the free subarrangements $\CA_\CI$ of 
$\CA(\Phi)$ are also inductively free
with possible exceptions only in type $E_8$.
The remaining instances in type $E_8$ 
were settled only recently in 
\cite{cuntzroehrleschauenburg:ideal}.

Note that if $\CI = \varnothing$, then 
$\CA_\CI = \CA(\Phi)$ is just the reflection arrangement of $\Phi$ and so 
$\CA_\varnothing$ is $K(\pi,1)$ by Deligne's result.
So we may assume that $\CI \ne \varnothing$.

Next we describe a combinatorial condition for 
an ideal $\CI \subseteq \Phi^+$ from \cite{roehrle:ideal}.
Using induction and Terao's fibration theorem \cite{terao:modular}, 
it allows us to show that a large class of 
arrangements of ideal type consists of $K(\pi,1)$ arrangements.
Let $\Phi_0$ be a (standard) parabolic subsystem of $\Phi$ 
and let 
\[
\Phi^c_0 := \Phi^+ \setminus \Phi_0^+,
\]
the set of  positive roots in the ambient root system which do 
not lie in the smaller one.

\begin{condition} 
[{\cite[Cond.~1.10]{roehrle:ideal}}]
\label{cond:linear}
Let $\CI\ne \varnothing$ be an ideal in $\Phi^+$ and let
$\Phi_0$ be a maximal parabolic subsystem of $\Phi$
such that $\Phi^c_0 \cap \CI^c \ne \varnothing$.
Assume that 
firstly, $\Phi^c_0 \cap \CI^c$ is linearly ordered with respect to 
$\preceq$ so that there is 
a unique root of every occurring 
height in $\Phi^c_0 \cap \CI^c$, 
and secondly, 
for any $\alpha \ne \beta$ 
in $\Phi^c_0 \cap \CI^c$, 
there is a $\gamma \in \Phi_0^+$ so that 
$\alpha, \beta$, and $\gamma$
are linearly dependent.
\end{condition} 

The instances when this condition is satisfied 
have been determined in \cite{roehrle:ideal}.

Our first main result shows that  
Condition \ref{cond:linear}
entails the $K(\pi,1)$-property
for the 
associated arrangement 
of ideal type $\CA_\CI$.

\begin{theorem}
\label{thm:main}
Let $\CI \ne \varnothing$ be an ideal in $\Phi^+$ and let
$\Phi_0$ be a maximal parabolic subsystem of $\Phi$
such that Condition \ref{cond:linear} is satisfied.
Then $\CA_\CI$ is $K(\pi,1)$.

Specifically, this is the case if and only if one of the following holds:
\begin{itemize}
\item[(i)]
$\Phi$ is of type
$A_n$, $B_n$, $C_n$, for $n \ge 2$ or $G_2$ and
$\CI$ is any ideal in $\Phi^+$;
\item[(ii)] $\Phi$ is of type
$D_n$, for $n \ge 4$ and either $\CI^c$ does not 
contain both  $e_{1} \pm e_{n}$, or 
$\CI$ is generated by the root $e_{n-2} + e_{n-1}$;
\item[(iii)]
$\Phi$ is of type $F_4$, $E_6$, $E_7$, or $E_8$
and $\CI$ is as in \cite[\S 4]{roehrle:ideal}.
\end{itemize}
\end{theorem} 

In addition we use Thom's first isotopy lemma
to construct explicit locally trivial fibrations in each of the
remaining 
instances  in type $D_n$ not covered in Theorem \ref{thm:main}(ii), 
i.e.~when $\Phi$ is of 
type $D_n$ and $\CI^c$ does 
contain both $e_{1} \pm e_{n}$.
Combined with 
Theorem \ref{thm:main},
this gives our second main result.

\begin{theorem}
\label{thm:classical}
For $\Phi$ of classical type and 
$\CI$ an ideal in $\Phi^+$, 
we have  that $\CA_\CI$ is $K(\pi,1)$.
\end{theorem} 

In Table \ref{table:cond} we present the number of all 
arrangements of ideal type for each exceptional type in the first row.
In the second row, we list the number of 
all $\CA_\CI$ when $\CI$ satisfies 
Condition \ref{cond:linear} with respect to
a suitable parabolic subsystem, cf.~\cite[Table 1]{roehrle:ideal}. 
Thus in these instances 
$\CA_\CI$ is $K(\pi,1)$, by Theorem \ref{thm:main}(iii). 

\begin{table}[ht!b]
\renewcommand{\arraystretch}{1.6}
\begin{tabular}{r|ccccc}
\hline
$\Phi$ & $E_6$ & $E_7$ & $E_8$ & $F_4$ & $G_2$ \\ 
\hline\hline
all $\CA_\CI$ & 833 & 4160 & 25080 & 105 & 8\\
aspherical  $\CA_\CI$ & 771 & 3433 & 18902 & 85 & 8\\
\hline
\end{tabular}
\smallskip
\caption{aspherical $\CA_\CI$ for exceptional $\Phi$ from Theorem \ref{thm:main}. }
\label{table:cond}
\end{table}

It is evident from Table \ref{table:cond} that 
with the possible exception of a relatively small number of
cases in the exceptional types, all $\CA_\CI$ 
are $K(\pi,1)$.
The number of possible exceptions
in type $F_4$, $E_6$, $E_7$, $E_8$ are 
20, 62, 727, respectively 6178.
Thus, Theorems \ref{thm:main} and \ref{thm:classical}
give strong evidence for the following conjecture.

\begin{conjecture}
\label{conj:main}
Let $\Phi$ be a reduced root system with 
Weyl arrangement $\CA(\Phi)$.
Then any subarrangement of
ideal type $\CA_\CI$ of $\CA(\Phi)$ is a $K(\pi,1)$-arrangement.
\end{conjecture}

\begin{remark}
\label{rem:f4}
(i).
Let $\Phi$ be of type $F_4$ and let $\CI$ be the ideal 
generated by the root $0122$ of height $5$. Although 
$\CI$  is not covered by Theorem \ref{thm:main}, it
turns out that $\CA_\CI$ is simplicial (see \cite{cuntzeckenberger:groupoid}), 
and so $\CA_\CI$ is $K(\pi,1)$.

(ii).
Since the $\CA_\CI$ in type $E_6$ and type $E_7$ are localizations of 
arrangements of ideal type in type $E_8$, thanks to Remark \ref{rem:local},
the open cases in Conjecture \ref{conj:main} reduce to 
the ones in type $F_4$ and $E_8$.
\end{remark}

\begin{remark}
\label{rem:main}
It is worth emphasizing that Theorems \ref{thm:main}  and \ref{thm:classical} provide new examples for 
$K(\pi,1)$-arrangements that are neither fiber-type, nor simplicial.
For instance, one can check that 
none of the non-supersolvable arrangements $\CA_\CI$ in type $E_6$ that
are shown to be $K(\pi,1)$ by Theorem \ref{thm:main} are simplicial.
See also Example \ref{ex:kpione}.

Note that in type $D_n$ and type $B_n$, some of the arrangements $\CA_\CI$
that contain the full braid arrangement of $A_{n-1}$ as a subarrangement
are shown to be $K(\pi,1)$ in \cite[\S 5]{falkproudfoot}.
\end{remark}

For general information about arrangements, Weyl groups and root systems,  
we refer the reader to \cite{bourbaki:groupes} and 
\cite{orlikterao:arrangements}.

\section{Preliminaries}
\label{sect:prelims}

\subsection{Hyperplane arrangements}
\label{ssect:arrangements}
Let $V = \BBC^n$ 
be an $n$-dimensional complex vector space.
A \emph{hyperplane arrangement} is a pair
$(\CA, V)$, where $\CA$ is a finite collection of hyperplanes in $V$.
Usually, we simply write $\CA$ in place of $(\CA, V)$.

The \emph{lattice} $L(\CA)$ of $\CA$ is the set of subspaces of $V$ of
the form $H_1\cap \ldots \cap H_i$ where $\{ H_1, \ldots, H_i\}$ is a subset
of $\CA$. 
For $X \in L(\CA)$, we have two associated arrangements, 
firstly
$\CA_X :=\{H \in \CA \mid X \subseteq H\} \subseteq \CA$,
the \emph{localization of $\CA$ at $X$}, 
and secondly, 
the \emph{restriction of $\CA$ to $X$}, $(\CA^X,X)$, where 
$\CA^X := \{ X \cap H \mid H \in \CA \setminus \CA_X\}$.
The lattice $L(\CA)$ is a partially ordered set by reverse inclusion:
$X \le Y$ provided $Y \subseteq X$ for $X,Y \in L(\CA)$.

Throughout, we only consider arrangements $\CA$
such that $0 \in H$ for each $H$ in $\CA$.
These are called \emph{central}.
In that case the \emph{center} 
$T(\CA) := \cap_{H \in \CA} H$ of $\CA$ is the unique
maximal element in $L(\CA)$  with respect
to the partial order.
A \emph{rank} function on $L(\CA)$
is given by $r(X) := \codim_V(X)$.
The \emph{rank} of $\CA$ 
is defined as $r(\CA) := r(T(\CA))$.

\subsection{$K(\pi,1)$-arrangements}
\label{ssect:kpionearrangements}

A member $X$ in $L(\CA)$ is said to be \emph{modular}
provided $X + Y \in L(\CA)$ for every $Y \in L(\CA)$,
 \cite[Cor.\ 2.26]{orlikterao:arrangements}.
The following is an immediate consequence of Terao's work 
\cite{terao:modular} (see also \cite[\S 5.5]{orlikterao:arrangements}).
Indeed, $\CA$ is strictly linearly fibered (see Definition \ref{def:strictlinfib})
if and only if $L(\CA)$ admits 
a modular element of rank $r-1$, see \cite[Cor.~2.14]{terao:modular} 
(cf.~\cite[Cor.~5.112]{orlikterao:arrangements}).

\begin{lemma}
\label{lem:modular}
Let $\CA$ be a complex arrangement of rank $r$.
Suppose that $X \in L(\CA)$ is modular of rank $r-1$.
If $\CA_X$ is $K(\pi,1)$, then so is $\CA$.
\end{lemma}

\begin{remark}
\label{rem:local}
Thanks to an observation by Oka, 
if the complex arrangement $\CA$ is $K(\pi, 1)$, 
then so is every localization $\CA_X$ for $X$ in $L(\CA)$, 
e.g., see \cite[Lem.~1.1]{paris:deligne}.
\end{remark}

There is a standard construction for $K(\pi,1)$-arrangements using 
locally trivial fibrations
with $K(\pi,1)$-spaces as bases and fibers. The long exact sequence 
in homotopy theory then gives that $\CM(\CA)$ is a $K(\pi,1)$-space, 
e.g.~see \cite[Thm.~5.9]{orlikterao:arrangements}.
We recall two basic 
definitions due to
Falk and Randell \cite{falkrandell:fiber-type};
also see \cite[Defs.~5.10, 5.11]{orlikterao:arrangements}.

\begin{defn}
\label{def:strictlinfib}
An $n$-arrangement $\CA$ is called \emph{strictly linearly fibered}
if, after a suitable linear change of coordinates, the restriction of the projection of $\CM(\CA)$ to the first $n-1$ coordinates is a 
locally trivial fibration whose base space is the complement of an arrangement in $\BBC^{n-1}$, and whose fiber is the
complex line $\BBC$ with finitely many points removed.
\end{defn}

\begin{defn}
\label{def:fibertype}
\begin{itemize}
\item[(i)]
The $1$-arrangement $(\{0\}, \BBC)$ is \emph{fiber-type}.
\item[(ii)]
For $n \ge 2$, the $n$-arrangement $\CA$ is \emph{fiber-type} if $\CA$ is strictly linearly
fibered with base $\CM(\CB)$, where $\CB$ is an $(n - 1)$-arrangement of fiber-type.
\end{itemize}
\end{defn}

A repeated application of the homotopy exact sequence shows that 
a fiber-type arrangement $\CA$ is $K(\pi, 1)$, e.g.~see
\cite[Prop.~5.12]{orlikterao:arrangements}.

The following important tool for proving that a given map is a locally trivial fibration 
is due to Thom \cite{thom}, see also \cite{mather:strat}.

\begin{theorem}[Thom's first isotopy lemma]\label{thm:thom}
Let $M$, $P$ be smooth manifolds, $f\colon M \to P$ a smooth mapping and $S \subseteq M$ a closed subset which admits a Whitney stratification $\FS$.
Suppose $f\vert_S\colon S \to P$ is proper and $f\vert_X\colon X \to P$ is a submersion for each stratum $X \in \FS$. Then $f\vert_S\colon S \to P$ is 
a locally trivial fibration and in particular $f\vert_X\colon X \to P$ is a locally trivial fibration for all $X \in \FS$.
\end{theorem}

Let $\CB_n$ be the reflection arrangement of the hyperoctahedral group of type
$B_n$. 
In the following example we consider a 
fiber-type subarrangement $\CJ_n$ of 
$\CB_n$ which 
is used in \S \ref{s:typeD} in the proof of 
Theorem \ref{thm:classical}.

\begin{example} 
\label{ex:kpione2}
The subarrangement $\CJ_n$ of $\CB_n$ is obtained by 
removing the anti-diagonals from $\CB_n$.
So  $\CJ_n$ is the union of the rank $n$ Boolean 
arrangement and the braid arrangement $\CA_{n-1}$,
i.e.~$\CJ_n$ has defining polynomial 
\[
Q(\CJ_n) := \prod\limits_{i = 1}^n x_i\prod\limits_{1 \leq i < j \leq n}\left(x_i - x_j\right).
\]
One easily checks that $\CJ_n$ is fiber-type, e.g.\ by projecting
onto the first $n - 1$ coordinates and using induction on $n$.

We observe that the fiber-type arrangement $\CJ_n$ was already used by  
Brieskorn in his proof of the asphericity of the Coxeter arrangement in type $D_n$, see \cite{brieskorn:tresses}
and \cite[\S 5]{falkrandell:fiber-type}. 
Also note that $\CJ_n$ is the irreducible version of the braid arrangement of type $A_n$. 
It is isomorphic to the restriction $\CA(A_n)^X$, where $X = \ker(x_0)$: the hyperplane $\ker x_i$ in $\CJ_n$
then corresponds to the hyperplane $\ker(x_0 - x_i)$ in $\CA(A_n)$.
\end{example} 

The following related example shows that in general a subarrangement 
of a Coxeter arrangement need not be $K(\pi,1)$ (nor free).

\begin{example} 
\label{ex:nonkpione}
Let $\CB_n$ be as above and let $\CA_{n-1}$ be its subarrangement consisting of the braid 
arrangement of type $A_{n-1}$.
Let 
\[
\CK_n  := \CB_n \setminus \CA_{n-1}
\]
be the complement of $\CA_{n-1}$ in $\CB_n$.
As opposed to the subarrangement $\CJ_n$ of 
$\CB_n$ from Example \ref{ex:kpione2}, 
rather than 
removing the anti-diagonal hyperplanes from $\CB_n$, 
for $\CK_n$ we 
remove all the diagonals instead.
Thus $\CK_n$ has defining polynomial 
\[
Q(\CK_n) = \prod\limits_{i = 1}^n x_i\prod\limits_{1 \leq i < j \leq n}\left(x_i + x_j\right).
\]
We show by induction on $n$
that $\CK_n$ is not $K(\pi,1)$ for $n \ge 3$.
Owing to \cite[(3.12)]{falkrandell:homotopy},
$\CK_3$ is not $K(\pi,1)$.
Now suppose that $n > 3$ and that the statement holds for $\CK_{n-1}$.
Let $X := \cap_{i = 1}^{n-1} \ker x_i$. Then one readily checks 
that 
\[
(\CK_n)_X \cong \CK_{n-1}. 
\]
It follows from our induction hypothesis and Remark \ref{rem:local} that 
also $\CK_n$ fails to be $K(\pi,1)$.

In \cite[(3.12)]{falkrandell:homotopy}, Falk and Randell also observe that 
$\CK_3$ is not free. Accordingly, by the 
argument above along with 
\cite[Thm.\ 4.37]{orlikterao:arrangements} we see 
that $\CK_n$ is not free for all $n \ge 3$.

So while the construction of $\CK_n$ is quite similar to 
that of $\CJ_n$, its 
combinatorial, algebraic and topological properties differ sharply from
those of $\CJ_n$.
\end{example}

\section{Proof of Theorem \ref{thm:main}}
\label{s:idealtype}

Let $\Phi$
be a reduced root system 
of rank $n$ with 
Weyl group $W$ and \emph{reflection arrangement}
$\CA = \CA(\Phi) = \CA(W)$.
Let $\Phi^+$
be the set of positive roots with respect to 
some set of simple roots $\Pi$  of $\Phi$.
For $\Pi_0$ a proper subset of $\Pi$, the 
\emph{(standard parabolic) subsystem} of $\Phi$
generated by $\Pi_0$ is $\Phi_0 := \BBZ \Pi_0 \cap \Phi$,
cf.~\cite[Ch.\ VI \S 1.7]{bourbaki:groupes}. 
Define $\Phi_0^+ := \Phi_0 \cap \Phi^+$,
the set of positive roots of $\Phi_0$ with respect to $\Pi_0$.
If the rank of $\Phi_0$ is $n-1$, 
then $\Phi_0$ is said to be \emph{maximal}.

Set $X_0 := \cap_{\gamma \in \Phi_0^+} H_\gamma$.
Then $\CA(\Phi)_{X_0} = \CA(\Phi_0)$. Therefore, 
the reflection arrangement $\CA(W_{X_0})$
of the parabolic subgroup $W_{X_0}$
is just $\CA(\Phi_0)$,
i.e.\ $\Phi_0$ is the root system of $W_{X_0}$
(cf.~\cite[Thm.\ 6.27, Cor.\ 6.28]{orlikterao:arrangements}).

\begin{defn}
\label{def:I1}
Fix a standard parabolic subsystem $\Phi_0$ of $\Phi$.
For $\CI$ an ideal in $\Phi^+$, 
\[
\CI_0 := \CI \cap \Phi_0^+
\]
is an ideal in $\Phi_0^+$. Thus 
\[
\CA_{\CI_0} := \{H_\gamma \mid \gamma \in \CI_0^c = \Phi_0^+ \setminus \CI_0\}
\] 
is an arrangement of ideal type in 
$\CA(\Phi_0)$, the Weyl arrangement of $\Phi_0$. 
\end{defn}

Obviously, since 
$\CI_0^c  = \Phi_0^+ \setminus \CI_0 = \CI^c \cap \Phi_0^+ \subseteq \CI^c$, 
we may view $\CA_{\CI_0}$ as a
subarrangement of $\CA_\CI$
rather than as a subarrangement of $\CA(\Phi_0)$.
Note however, as such, 
$\CA_{\CI_0}$ is not of ideal type in $\CA$ in general,  
since $\CI_0$ need not be an ideal in $\Phi^+$.
We continue by recalling some basic facts from 
\cite{roehrle:ideal}.

\begin{lemma}
[{\cite[Lem.~3.1]{roehrle:ideal}}]
\label{lem:ideallocal}
Viewing $\CA_{\CI_0}$ as a subarrangement of $\CA_\CI$,
we have $\CA_{\CI_0} = (\CA_\CI)_{X_0}$.
\end{lemma}

The next observation shows that  
Condition \ref{cond:linear} entails the presence 
of a modular element in $L(\CA_\CI)$ of rank 
$r(\CA_\CI)-1$.

\begin{lemma}
[{\cite[Lem.~3.4]{roehrle:ideal}}]
\label{lem:condition-modular}
If $\CI \subseteq \Phi^+$ and 
$\Phi_0$ satisfy Condition \ref{cond:linear},
then the center 
$Z:= T((\CA_\CI)_{X_0})$
of $(\CA_\CI)_{X_0}$
is modular of rank $r(\CA_\CI)-1$ in $L(\CA_\CI)$.
\end{lemma}

Observe that $X_0$ itself need not belong to $L(\CA_\CI)$, 
e.g.~see \cite[Ex.~3.3]{roehrle:ideal}.

Our next result
shows that Condition \ref{cond:linear}
allows us to derive
the $K(\pi,1)$ property for $\CA_\CI$ from that of 
$\CA_{\CI_0}$. 
It is just a consequence of Lemma \ref{lem:modular}.

\begin{corollary}
\label{cor:I1I}
Let $\CI$ be an ideal in $\Phi^+$ and let
$\Phi_0$ be a maximal parabolic subsystem of $\Phi$
such that either $\Phi^c_0 \cap \CI^c = \varnothing$
or else Condition \ref{cond:linear} is satisfied.
Then $\CA_{\CI_0}$ is $K(\pi,1)$ 
if and only if $\CA_\CI$ is $K(\pi,1)$.
\end{corollary}

\begin{proof}
If $\Phi^c_0 \cap \CI^c = \varnothing$, 
then $\CA_\CI$ is the product of the empty 1-dimensional arrangement and 
$\CA_{\CI_0}$, and so the result is clear.
Else, $\CA_{\CI_0} = (\CA_\CI)_{X_0} = (\CA_\CI)_Z$, 
by Lemmas \ref{lem:ideallocal} and \ref{lem:condition-modular}.
Therefore, the forward implication follows from Lemmas \ref{lem:modular} and
\ref{lem:condition-modular}, while 
the reverse implication is clear by 
Remark \ref{rem:local}.
\end{proof}

We note that modular elements of corank 1 were constructed in 
\cite[Lem.~5.4]{falkproudfoot}
for certain subarrangements of 
the reflection arrangement $\CB_n$
of the hyperoctahedral group of type
$B_n$ that contain the full braid arrangement $\CA_{n-1}$ of type $A_{n-1}$. 

\begin{remark}
\label{rem:typeD}
Let $\Phi$ be of type $D_n$, for $n \ge 4$ and
let $\Phi_0$ be the standard subsystem of $\Phi$ of type $D_{n-1}$.
Here and in \S \ref{s:typeD} we use the notation for the positive roots
from \cite[\S 4.8, Planche IV]{bourbaki:groupes}.
Then $\Phi^c_0 = \{ e_1 \pm e_j \mid 2 \le j \le n\}$.
Note that $\Phi^c_0$ is not linearly ordered by 
$\preceq$, for $e_1 \pm e_n$ both have height $n-1$.

Suppose that $\CI \ne \varnothing$ fails to satisfy Condition \ref{cond:linear}
(with respect to our fixed $\Phi_0$).
This is precisely the case when both $e_1 \pm e_n$ belong to $\CI^c$.
Then $\CI$ consists of roots from $\Phi^+$ each of which 
admits the root $e_{n-2} + e_{n-1}$ of height $3$ as a summand.
Otherwise, at least 
one of $e_1 \pm e_n$ must belong to  $\CI$, as $\CI$ is an ideal in $\Phi^+$.
This contradicts the assumption on $\CI$.
In turn this implies that 
if $\CI_0 = \Phi_0 \cap \CI$ is non-empty and 
fails to satisfy Condition \ref{cond:linear}
with respect to the maximal rank subsystem of $\Phi_0$ of type $D_{n-2}$,
then $\CI$ fails to satisfy Condition \ref{cond:linear}
with respect to $\Phi_0$.
For, if each root in $\CI_0$ 
admits the root $e_{n-3} + e_{n-2}$ as a summand, 
then necessarily each root in $\CI$
has $e_{n-2} + e_{n-1}$ as a summand.

We conclude that if 
$\CI$ satisfies Condition \ref{cond:linear}
with respect to $\Phi_0$, then
$\CI_0$ satisfies Condition \ref{cond:linear}
with respect to the subsystem of $\Phi_0$ of type $D_{n-2}$. 
\end{remark}

\begin{proof}[Proof of Theorem \ref{thm:main}]
(i).
For $\Phi$ of type $A_n$, $B_n$, or $C_n$ for $n \ge 2$, 
it follows from \cite[\S 7]{sommerstymoczko} that 
for $\Phi_0$ the canonical maximal rank subsystem of type
$A_{n-1}$, $B_{n-1}$, or $C_{n-1}$, respectively, 
each $\CI$ satisfies 
Condition \ref{cond:linear}, because irrespective of $\CI$, in each case
$\Phi^c_0$ is linearly ordered by $\preceq$. 
So the result follows in this instance
from induction on the rank, Corollary~\ref{cor:I1I},
and the fact that central rank $2$-arrangements are $K(\pi,1)$,
cf.~\cite[Prop.\ 5.6]{orlikterao:arrangements}.
The last result also implies that 
for $\Phi$ of type $G_2$
each arrangement of ideal type is $K(\pi,1)$.
The very same inductive argument shows that
in all these cases each $\CA_\CI$ is actually supersolvable, see
\cite[Thm.\ 1.5]{roehrle:ideal}; see also \cite[Thms.~6.6, 7.1]{hultman:koszul}
where this is proved by different means.

(ii).
Now let $\Phi$ be of type $D_n$, for $n \ge 4$ and
let $\Phi_0$ be the standard subsystem of $\Phi$ of type $D_{n-1}$.
We argue by induction on $n$.
For $n = 4$,  the result follows from \cite[Lem.\ 6.1]{roehrle:ideal}.
Indeed, each $\CA_\CI$ which satisfies the hypothesis of the theorem is 
already supersolvable.

Now suppose that $n \ge 5$ and that the result holds for
root systems of type $D$ of smaller rank.
If $\CI_0 = \Phi_0 \cap \CI = \varnothing$, 
then $\CA_{\CI_0} = \CA(D_{n-1})$. Being simplicial, the latter is  $K(\pi,1)$.
It follows from Corollary~\ref{cor:I1I} that also $\CA_\CI$ is  $K(\pi,1)$.

Now suppose that $\CI_0 \ne \varnothing$.
By Remark \ref{rem:typeD}, 
$\CI_0$ satisfies Condition \ref{cond:linear} and 
so by induction, 
$\CA_{\CI_0}$ is $K(\pi,1)$. Using 
Corollary~\ref{cor:I1I} again, 
we conclude that $\CA_\CI$ is also $K(\pi,1)$,
as desired.

Now let $\CI$ be the ideal in $\Phi$ 
which is generated by 
$e_{n-2} + e_{n-1}$. Then one easily checks that 
$\CI$ satisfies Condition \ref{cond:linear}
with respect to either one of the two subsystems of type 
$A_{n-1}$, see \cite[Ex.\ 3.9]{roehrle:ideal}. 
So it follows from part (i) and Corollary~\ref{cor:I1I} that 
$\CA_\CI$ is also $K(\pi,1)$ in this instance.

(iii). 
Now suppose that $\Phi$ is of type $F_4$, $E_6$, $E_7$, or $E_8$.
All instances when $\CI$ satisfies Condition \ref{cond:linear}
with respect to a suitably chosen maximal rank subsystem $\Phi_0$ 
are discussed in detail in \cite[\S 4]{roehrle:ideal}.
Perusing the arguments and in particular the 
data in Tables 6 - 9 in \cite[\S 4]{roehrle:ideal}, 
one checks that in each instance either  
$\CI_0 = \varnothing$, or else 
$\CI_0 \ne \varnothing$ satisfies Condition \ref{cond:linear} 
with respect to $\Phi_0^+$. 
In the first instance we have
$\CA_{\CI_0} = \CA(\Phi_0)$ which is simplicial, and so it is $K(\pi,1)$.
In the second instance 
$\CA_{\CI_0}$ is $K(\pi,1)$ by induction. 
In both cases 
it follows from Corollary~\ref{cor:I1I} that 
also $\CA_\CI$ 
is $K(\pi,1)$, as claimed.
\end{proof}

We illustrate the inductive arguments in the proof of 
Theorem \ref{thm:main}(iii) 
in the following examples.

\begin{example}
\label{ex:kpione}
(a).
Let $\Phi$ be of type $E_6$ and let $\CI$ be the ideal generated by
the root $\stackrel{00111}{_{0}}$ of height $3$.
Then according to the last entry for $E_6$ in \cite[Table 6]{roehrle:ideal},
$\CI$ together with the subsytem $\Phi_0$ of type $D_5$ satisfy 
Condition \ref{cond:linear}.
Since $\CI_0 = \varnothing$,
$\CA_{\CI_0} = \CA(\Phi_0)$ is the full reflection arrangement of type $D_5$
which is $K(\pi,1)$.
Thus so is $\CA_\CI$, by Corollary~\ref{cor:I1I}.

(b). 
Next consider $\Phi$ of type $E_7$ and let $\CI$ be the ideal generated by
the root $\stackrel{001110}{_{0\ }}$ of height $3$.
Then according to the next to last entry for $E_7$ in \cite[Table 6]{roehrle:ideal},
$\CI$ together with the subsystem $\Phi_0$ of type $E_6$ satisfy 
Condition \ref{cond:linear}.
Now $\CI_0$ is just the ideal in $E_6$  considered in part (a).
Consequently, 
$\CA_{\CI_0}$ is $K(\pi,1)$.
But then so is $\CA_\CI$, again by Corollary~\ref{cor:I1I}.

(c). 
Finally, let $\Phi$ be of type $E_8$ and let $\CI$ be the ideal generated by
the root $\stackrel{0011100}{_{0\ \ }}$ of height $3$.
Thanks to the data in the fifth row for $E_8$ in \cite[Table 6]{roehrle:ideal},
$\CI$ together with the subsytem $\Phi_0$ of type $E_7$ satisfy 
Condition \ref{cond:linear}.
As $\CI_0$ is the ideal in $E_7$ considered in part (b),
we have that $\CA_{\CI_0}$ is $K(\pi,1)$ and so is
$\CA_\CI$, thanks to Corollary~\ref{cor:I1I}.
\end{example}

Note that none of the three arrangements of ideal type $\CA_\CI$ considered in Example
\ref{ex:kpione} is supersolvable (cf.~\cite[Lem.~6.2]{hultman:koszul}) and none of them is simplicial.

\section{Proof of Theorem \ref{thm:classical}}
\label{s:typeD}

Thanks to Theorem \ref{thm:main}, 
Theorem \ref{thm:classical} follows once the outstanding
instances in Type $D_n$ not covered in Theorem \ref{thm:main}(ii) are resolved.
Accordingly, these are the instances when
$\CI$ consists of roots from $\Phi^+$ each of which 
admits the root $e_{n-2} + e_{n-1}$ of height $3$ as a summand,
by Remark \ref{rem:typeD}.
In addition, by the proof of Theorem \ref{thm:main}, 
we need not consider the case when 
$\CI$ is the ideal in $\Phi$ which is generated by 
$e_{n-2} + e_{n-1}$.
We list the different cases we need to consider below.
We distinguish three different types of such ideals 
$\CI$ according to their 
generators. In the first two instances, each $\CI$ is 
generated by just a single root and by two in the third case:\\

\begin{itemize}
\item[(I)] $0\ldots01\ldots1\!\stackrel{1}{_{1}} \ = e_r+e_{n-1}$ for $1 \le r < n-2$. 
Here $r$ is the first position with $1$ as coefficient.

\item[(II)] $0\ldots01\ldots12\ldots12\!\stackrel{1}{_{1}} \ =  e_s+e_t$, where $1 \le  s < t < n-1$. 
Here $s$ is the first position with a coefficient $1$ and $t$ is the first position labeled with $2$.

\item[(III)] $0\ldots01\ldots1\!\stackrel{1}{_{1}} \ = e_r+e_{n-1}$ for $1 \le r < n-2$ and 
$0\ldots01\ldots12\ldots12\!\stackrel{1}{_{1}} \ =  e_s+e_t$, where $1 \le  s < t < n-1$ and $r < s$.
Note that the two roots are not comparable,  since $r < s$.
\end{itemize}

In the following we give explicit locally trivial fibrations of the complements in each of the three cases above.
First, we consider spaces that are going to serve as our bases for the locally trivial fibrations
in these three instances. 
Recall the fiber-type subarrangement $\CJ_n$ of $\CB_n$ from  
Example \ref{ex:kpione2}. In the following three lemmas, we exhibit three classes of subarrangements of 
$\CJ_n$ that are still fiber-type.  

\begin{lemma}\label{lem:type_a}
For $1 \leq r < n - 1$ fixed, the $n$-arrangement
$$
\CJ_n(r) := \CJ_n\setminus\left\{\ker\left(x_i - x_j\right)\ \middle|\ 1 \leq i \leq r < j \leq n\right\}
$$
is fiber-type.
\begin{proof}
We distinguish two cases: First, assume $r = 1$. Then the projection $\pi\colon\BBC^n\to\BBC^{n-1}$,
$\left(z_1, \ldots, z_n\right) \mapsto \left(z_2, \ldots, z_n\right)$ induces a locally trivial fibration
$\widetilde{\pi}\colon\CM\left(\CJ_n(r)\right)\to\CM\left(\CJ_{n-1}\right)$ with fiber the complex plane with
one point removed.

Now assume that $r > 1$. Then we have $\CJ_n(r) = \CJ_r \times \CJ_{n-r}$.

Thus in both cases, $\CJ_n(r)$ is fiber-type.
\end{proof}
\end{lemma}

\begin{lemma}\label{lem:type_b}
For $1 \leq s < t < n$ fixed, the $n$-arrangement
$$
\CJ_n(s,t) := \CJ_n\setminus\left\{\ker\left(x_i - x_j\right)\ \middle|\ 1 \leq i \leq s < j \leq t\right\}
$$
is fiber-type.
\begin{proof}
As in the proof of Lemma \ref{lem:type_a}, let $\pi\colon\BBC^n\to\BBC^{n-1}$ be the projection
$\left(z_1, \ldots, z_n\right) \mapsto \left(z_2, \ldots, z_n\right)$.
First, assume $s = 1$. Then $\pi$ induces a locally trivial fibration 
\[
\widetilde{\pi}\colon\CM\left(\CJ_n(1,t)\right) \to\CM\left(\CJ_{n-1}\right)
\]
with fiber the complex plane with $n - t + 1$ points removed. So  
$\CJ_n(1,t)$ is fiber-type.

Now assume $s > 1$. Then $\pi$ induces a locally trivial fibration 
\[
\widetilde{\pi}\colon\CM\left(\CJ_n(s,t)\right)\to \CM\left(\CJ_{n-1}(s-1,t-1)\right) 
\]
with fiber the complex plane with $n - t + s$ points removed. Thus $\CJ_n(s,t)$
is fiber-type by induction on $s$.
\end{proof}
\end{lemma}

\begin{lemma}\label{lem:type_c}
For $1 \leq r < s < t < n$ fixed,  the $n$-arrangement
$$
\CJ_n(r,s,t) := \CJ_n\setminus\left\{\ker\left(x_i - x_j\right)\ \middle|\ 1 \leq i \leq r < j \leq n \quad\text{or}\quad
r < i \leq s < j \leq t\right\}
$$
is fiber-type.
\begin{proof}
Take $\pi\colon\BBC^n\to\BBC^{n-1}$ to be the projection
$\left(z_1, \ldots, z_n\right) \mapsto \left(z_1, \ldots, z_{s-1}, z_{s+1}, \ldots, z_n\right)$.
In case $s > r + 1$, this projection induces a locally trivial fibration 
\[
\widetilde{\pi}\colon\CM\left(\CJ_n(r,s,t)\right) \to \CM\left(\CJ_{n-1}(r, s-1, t-1)\right).
\]
If $s = r + 1$, it induces a locally trivial fibration 
\[
\widetilde{\pi}\colon\CM\left(\CJ_n(r,s,t)\right) \to \CM\left(\CJ_{n-1}(r)\right). 
\]
In both cases the fiber is the complex plane with $n - r + s - t + 1$ points removed.
Now the result follows by induction on $s$ and Lemma \ref{lem:type_a}.
\end{proof}
\end{lemma}

We observe that the identification of $\CJ_n$ with a braid arrangement mentioned in Example \ref{ex:kpione2} yields alternative proofs of
Lemmas \ref{lem:type_a} - \ref{lem:type_c} via Stanley's Theorem  \cite[Prop.~2.8]{stanley:super}.
For, the subarrangement $\CJ_n(r)$ corresponds to the
graphic arrangement with underlying graph the union of the complete subgraphs on the vertices $\{0, 1, \ldots, r\}$
and $\{0, r + 1, \ldots, n\}$. Further, $\CJ_n(s,t)$ corresponds to the union of the complete subgraphs on the vertices
$\{0, \ldots, s, t + 1, \ldots, n\}$ and $\{0, s + 1, \ldots, t, t + 1, \ldots, n\}$. 
The arrangement $\CJ_n(r,s,t)$ then corresponds to the union of complete subgraphs on the vertices
$\{0, 1, \ldots, r\}$, $\{0, r + 1, \ldots, s, t + 1, \ldots, n\}$ and $\{0, s + 1, \ldots, t, t + 1, \ldots, n\}$.
In all cases the graph is clearly chordal, so
the arrangement is fiber-type, thanks to   \cite[Prop.~2.8]{stanley:super}. 
\bigskip

Now let $\CI$ be of type (I), (II) or (III) listed above, 
set $\CA = \CA_\CI$ and in types (I) - (III) let $\CB$ be $\CJ_{n-1}(r)$, $\CJ_{n-1}(s,t)$ or
$\CJ_{n-1}(r,s,t)$, respectively. 
Consider the map
\begin{equation}
\label{eq:deff}
f\colon \CM(\CA) \to \CM(\CB) \quad \text{given by}\quad
(y_1, \ldots, y_n) \mapsto (y_n^2 - y_1^2, \ldots, y_n^2 - y_{n-1}^2).
\end{equation}

Note that in case $\CI = \varnothing$, i.e.~$\CA_\CI = \CA(\Phi)$, and $\CB = \CJ_n$, the map $f$ was  
used in \cite{brieskorn:tresses} to show asphericity in type $D_n$, see also \cite[\S 5]{falkrandell:fiber-type}.
Our argument that the map $f$ in \eqref{eq:deff} is a fibration over these larger bases is inspired by
an argument due to Li Li \cite{LiLi} who worked out the details
of Brieskorn's approach  \cite{brieskorn:tresses}. 

Set $Y := \CM(\CA)$ and $Z := \CM(\CB)$.
We can embed $Y$ into $\BBP^n \times Z$ by the ``graph'' map $\iota\colon Y \to \BBP^n \times Z$ defined by
$$\left(y_1, \ldots, y_n\right) \mapsto \left((1 :y_1: \ldots :y_n),f(y_1, \ldots, y_n)\right)$$
and denote the image of $Y$ by $C := \iota(Y)$. Then the map $f$ is just $f = \pi\vert_C \circ \widetilde{\iota}$, where
$\widetilde{\iota}\colon Y \to C$ is the homeomorphism induced by $\iota$ and $\pi\vert_C$ is the restriction of the projection
$\pi\colon \BBP^n \times Z \to Z$ to $C$. Thus $f$ is a locally trivial fibration if and only if $\pi\vert_C$ is one.

Now let $S_i$ be the hypersurface in $\BBC^n \times Z \subset \BBP^n \times Z$ defined by $z_i = y_n^2 - y_i^2$,
so that $C = S_1 \cap \ldots \cap S_{n - 1}$.
For $z = \left(z_1, \ldots, z_{n - 1}\right) \in Z$, let
$$\left(S_i\right)_z := S_i \cap \left(\BBC^n \times \left\{z\right\}\right) \subset \BBP^n \times \left\{z\right\}\quad\text{and}$$
$$C_z := \left(S_1\right)_z \cap \ldots \cap \left(S_{n - 1}\right)_z,$$
i.e. $C_z$ is the fiber of $\pi\vert_C$ over $z$. 
Moreover, let $\overline{C}$ and $\overline{C_z}$ denote the projective closure of $C$ and $C_z$ in $\BBP^n \times Z$, respectively.
Then
$$\overline{C} = \overline{S_1} \cap \ldots \cap \overline{S_{n-1}}\quad\text{and}$$
$$\overline{C_z} = \overline{\left(S_1\right)_z} \cap \ldots \cap \overline{\left(S_{n - 1}\right)_z},$$
where $\overline{S_i}$ is the hypersurface in $\BBP^n \times Z$ given by $z_iy_0^2 = y_n^2 - y_i^2$ and for
$z = \left(z_1, \ldots, z_{n - 1}\right) \in Z$,
$$\overline{\left(S_i\right)_z} := \overline{S_i} \cap \left(\BBP^n \times \left\{z\right\}\right).$$

Since $\overline{S_i}$ is defined by $y_n^2 - y_i^2 = z_i y_0$ for all $1 \leq i \leq n-1$ and the points at infinity are given 
by setting $y_0 = 0$, we get that $\overline{C_z}$ has the following
$2^{n-1}$ points at infinity:
$$\left(\left(0 : \pm1 : \ldots : \pm1 : 1\right), \left(z_1, \ldots, z_{n-1}\right)\right).$$

\begin{lemma}\label{lem:smoothfiber}
For each $z \in Z$, the projective closure $\overline{C_z}$ of $C_z$
is a smooth curve.
\begin{proof}
The $\overline{\left(S_i\right)_z}$ intersect transversally, which can be seen by looking at the Jacobian $J = \left(\frac{\partial f_j}{\partial t_i}(y)\right)$
of the polynomials given by
$$f_i\colon \overline{Y} \to \BBC, \quad \left(t_0 : t_1 : \ldots : t_n\right) \mapsto t_n^2 - t_i^2 - z_it_0^2,$$
where $\overline{Y}$ is the projective closure of $Y$ in $\BBP^n$.
\end{proof}
\end{lemma}

Moreover, we have the following:
\begin{lemma}
For each $z \in Z$, $\overline{C_z}$ is connected.
\begin{proof}
Every point in $\overline{C_z}$ satisfies the equations
$$\frac{y_n^2 - y_1^2}{z_1} = \ldots = \frac{y_n^2 - y_{n-1}^2}{z_{n-1}} = y_0^2.$$
First take $U_n$ to be the subset of $\overline{C_z}$ consisting of points $\left(\left(y_0 : \ldots : y_n\right), \left(z_1, \ldots, z_{n-1}\right)\right)$
with $y_n \neq 0$.
Thus considering the change of coordinates $x_i := \frac{y_i}{y_n}$ and fixing some $1 \leq j \leq n-1$, we get that
$$x_i^2 = g_i^j(x_j)\quad\text{for all}\quad 1 \leq i \leq n-1\quad\text{and}\quad x_0^2 = g_0^j(x_j),$$
where $g_i^j(x) = \frac{z_i}{z_j}x^2 + \frac{z_j-z_i}{z_j}$ and $g_0^j(x) = -\frac{1}{z_j}x^2 + \frac{1}{z_j}$. Let $\alpha_0$ and $\alpha_1$  be the two branches
of $y = x^2$. Then for any point $p \in U_n$ there are indices $k_i \in \left\{0, 1\right\}$ such that
\begin{multline*}p = \left(\left(\alpha_{k_0}(g_0^j(x_j)) : \ldots : \alpha_{k_{j-1}}(g_{j-1}^j(x_j)) : x_j : \right.\right. \\
\left.\left. \alpha_{k_{j+1}}(g_{j+1}^j(x_j)) : \ldots : \alpha_{k_{n-1}}(g_{n-1}^j(x_j)) : 1\right), \left(z_1, \ldots, z_{n-1}\right)\right).\end{multline*}
So by choosing an appropriate path in $\BBC$, we may path-connect $p$ to one of the points at infinity
$\left(\left(0 : \pm1 : \ldots : \pm1 : 1\right), \left(z_1, \ldots, z_{n-1}\right)\right)$. As $1 \leq j \leq n-1$ is arbitrary and $g_i^j(x) = g_i^j(-x)$, any
point $p \in U_n$ is path-connected to the point $\left(\left(0 : 1 : \ldots : 1\right), \left(z_1, \ldots, z_{n-1}\right)\right)$.

Now take $U_1$ to be the subset of $\overline{C_z}$ consisting of points $\left(\left(y_0 : \ldots : y_n\right), \left(z_1, \ldots, z_{n-1}\right)\right)$ with
$y_1 \neq 0$ and observe that $U_1 \cup U_n = \overline{C_z}$. By a similar argument as the one above, for any point $q \in U_1$ there are indices $k_i \in \left\{0, 1\right\}$ such that
$$q = \left(\left(\alpha_{k_0}(h_0(x_n)) : 1 : \alpha_{k_2}(h_2(x_n)) : \ldots : \alpha_{k_{n-1}}(h_{n-1}(x_n)) : x_n\right), \left(z_1, \ldots, z_{n-1}\right)\right),$$
where $h_0(x) = \frac{1}{z_1}x_n^2 - \frac{1}{z_1}$, $h_i(x) = \frac{z_1 - z_i}{z_1}x^2 + \frac{z_i}{z_1}$ and $x_i = \frac{y_i}{y_1}$.
Now we can again choose a path in $\BBC$ that connects $q$ to one of the points at infinity $\left(\left(0 : \pm1 : \ldots : \pm1 : 1\right), \left(z_1, \ldots, z_{n-1}\right)\right)$.
Thus, $\overline{C_z}$ is connected.
\end{proof}
\end{lemma}

Note that this also proves that $C_z$ is connected:\ as two points in $C_z$ are connected by a path through finitely many points at infinity and $\overline{C_z}$ is locally homeomorphic
to $\BBC$, we can alter the path around each of the points at infinity to get a path that completely lies inside $C_z$.

The above lemmas prove the following:
\begin{corollary}\label{cor:riemann}
For each $z \in Z$, the curve
$\overline{C_z}$ is a connected Riemann surface and $C_z$ is a connected Riemann surface with $2^{n-1}$ puncture points.
\end{corollary}

\begin{theorem}
The map $f$ defined in \eqref{eq:deff} is a locally trivial fibration.
\begin{proof}
Set $D = \overline{C}\setminus C$, the intersection of $\overline{C}$ with the infinity hyperplane. Then $\FS = \left\{C, D\right\}$ is a Whitney stratification of $\overline{C}$:\ it
is obviously locally finite and satisfies the condition of the frontier and as $C$ is open and $D$ its boundary, $\FS$ trivially satisfies Whitney condition $(B)$. The intersection of $D$
with a fiber $\BBP^n \times \left\{z\right\}$ of the projection $\pi$ is just the set of the $2^{n-1}$ points $\left(\left(0 : \pm 1 : \ldots : \pm 1 : 1\right), \left(z_1, \ldots, z_{n-1}\right)\right)$,
which we can think of locally as $2^{n-1}$ sections of $\pi$. Thus $\pi\vert_D$ is locally homeomorphic and therefore it is a submersion. The map $\pi\vert_C$ is a submersion as well,
which can be seen by considering the Jacobian again. Moreover, $\pi\vert_{\overline{C}}$ is proper, as $\overline{C}$ is a closed subset of $\BBP^n \times Z$ and $\pi$ is proper. Now
using Thom's first isotopy lemma, Theorem \ref{thm:thom}, $\pi\vert_{\overline{C}}$ is a locally trivial fibration and, in particular, $f = \pi\vert_C \circ \widetilde{\iota}$ is a fibration as well.
\end{proof}
\end{theorem}

This proves the following:

\begin{theorem}
If $\CI$ is of type (I), (II) or (III), then $\CA_\CI$ is $K(\pi, 1)$.
\begin{proof}
Consider the map $f\colon Y \to Z$ from \eqref{eq:deff}. Clearly, the fiber $f^{-1}(z)$ is homeomorphic to $C_z$, so by Corollary \ref{cor:riemann} it is a connected Riemann surface with $2^{n-1}$ puncture
points. Thus by the Uniformization Theorem, it is a $K(\pi, 1)$-space. By Lemmas \ref{lem:type_a}, \ref{lem:type_b} and \ref{lem:type_c}, $Z$ is a $K(\pi, 1)$-space as well. This
proves the theorem.
\end{proof}
\end{theorem}

This concludes the proof of Theorem \ref{thm:classical}. Note that none of the arrangements of ideal type $\CA_\CI$ of types (I) - (III) considered here 
are supersolvable (cf.~\cite[Lem.~6.2]{hultman:koszul}) and none of them are simplicial.
So these families of $\CA_\CI$ also provide new classes of $K(\pi,1)$-arrangements.

\begin{remark}\label{rem:fundgrpprod}
(i). If $\CA$ is strictly linearly fibered over $\CB$, then there always exists a section of the associated fibration of the complements 
$\CM(\CA) \to \CM(\CB)$, 
e.g.~see \cite[Cor.~1.1.6]{cohen}. 
As a consequence, by the splitting lemma, $\pi_1(\CM(\CA))$ is a semidirect product of $\pi_1(\CM(\CB))$ 
acting on the fundamental group of the fiber.
In particular, this applies to each of the cases considered in  Theorem \ref{thm:main}.

(ii). One can also construct a cross section to the fibration
$f\colon Y \to Z$ used in the proof of Theorem  \ref{thm:classical} as follows.
Let
$$y_n = y_n\left(z_1, \ldots, z_{n-1}\right) = \sqrt{\left\vert z_1 \right\vert + \ldots + \left\vert z_{n-1} \right\vert}.$$
Now for all $\left(z_1, \ldots, z_{n-1}\right) \in Z$, for all $1 \leq i \leq n - 1$ the real part of $y_i^2 = y_n^2 - z_i$ is positive. Thus choosing a branch $\alpha$ of the square root, we can define
$y_i = \alpha\left(y_n^2 - z_i\right)$ continuously, yielding a cross section $s\colon Z \to Y$.
This section was initially constructed by Falk and Randell in \cite[\S 5]{falkrandell:fiber-type}
in case $\CA$ is the full reflection arrangement of type $D_n$ which is strictly linearly fibered over  
$\CB = \CJ_{n-1}$, cf.~Example \ref{ex:kpione2}.
See also \cite[\S 1.1]{LeibmanMarkushevich} for a locally trivial fibration in this case with
a slightly different section.

As $f \circ s = \id_Z$, the short exact sequence of fundamental groups splits. Thus by the splitting lemma we see that $\pi_1(Y)$ is a semidirect product of $\pi_1(Z)$ acting on $\pi_1(C_z)$, where $C_z$
is the fiber over $z \in Z$ as above.
\end{remark}


\bigskip {\bf Acknowledgments}: 
We are grateful to Li Li for helpful discussions concerning his argument in
\cite{LiLi} and to Graham Denham and Michael Falk for providing 
the reference \cite[(3.12)]{falkrandell:homotopy},
where the non-$K(\pi,1)$-arrangement $\CK_3$ used in 
Example \ref{ex:nonkpione} seems to appear first in 
the literature.
We would also like to thank Daniel Cohen for raising 
the question about the existence of the 
sections of the fibrations discussed in Remark \ref{rem:fundgrpprod}.
Thanks are also due to the anonymous referee for making numerous suggestions improving the paper.

The research of this work was supported by 
DFG-grant RO 1072/16-1.


\bigskip

\bibliographystyle{amsalpha}

\begin{thebibliography}{ABC{\etalchar{+}}16}

\bibitem[ABC{\etalchar{+}}16]{abeetall:weyl}
T.~Abe, M.~Barakat, M.~Cuntz, T.~Hoge, and H.~Terao,
\emph{The freeness of ideal subarrangements of Weyl arrangements},
JEMS,  \textbf{18} (2016), no. 6, 1339--1348.

\bibitem[BC12]{barakatcuntz:indfree} M. Barakat and M. Cuntz, \emph{Coxeter and 
    crystallographic arrangements are inductively free}, Adv. Math \textbf{229}
    (2012), 691--709.

\bibitem[Br73]{brieskorn:tresses}
E.~Brieskorn,
\newblock Sur les groupes de tresses [d'apr\`es {V}. {I}. {A}rnold].
\newblock In {\em S\'eminaire Bourbaki, 24\`eme ann\'ee (1971/1972), Exp. No.
  401}, pages 21--44. Lecture Notes in Math., Vol. 317, Springer, Berlin, 1973.

\bibitem[Bou68]{bourbaki:groupes} N.~Bourbaki, \emph{\'{E}l\'ements de
    math\'ematique. {G}roupes et alg\`ebres de {L}ie. {C}hapitre {IV}-{VI}},
  Actualit\'es Scientifiques et Industrielles, No. 1337, Hermann, Paris,
  1968.

\bibitem[Co01]{cohen}
D.C.~Cohen, 
\emph{Monodromy of fiber-type arrangements and orbit configuration spaces}. 
Forum Math. \textbf{13} (2001), no. 4, 505--530. 

\bibitem[CH15]{cuntzeckenberger:groupoid}
 M.~Cuntz and I.~Heckenberger, 
\emph{Finite Weyl groupoids}. J. Reine Angew. Math. \textbf{702} (2015), 77--108.

\bibitem[CRS17]{cuntzroehrleschauenburg:ideal}
M. Cuntz, G.~R\"ohrle, and A.~Schauenburg,
\emph{Arrangements of ideal type are inductively free},
\url{https://arxiv.org/abs/1711.09760}

\bibitem[Del72]{deligne}
P.~Deligne,
\emph{Les immeubles des groupes de tresses g\'en\'eralis\'es},  
Invent. Math. \textbf{17} (1972), 273–-302.

\bibitem[FN62]{fadellneuwirth}
E.~Fadell and L.~Neuwirth, 
\emph{Configuration spaces}, Math. Scand. \textbf{10} (1962) 111--118.

\bibitem[FP02]{falkproudfoot}
M.~Falk and N.~Proudfoot, 
\emph{Parallel connections and bundles of arrangements}, Arrangements in Boston: a Conference on Hyperplane Arrangements (1999). 
Topology Appl. \textbf{118} (2002), no. 1-2, 65--83.

\bibitem[FR85]{falkrandell:fiber-type}
M.~Falk and R.~Randell, 
\emph{The lower central series of a fiber-type arrangement},
Invent. Math. \textbf{82} (1985), no. 1, 77--88. 

\bibitem[FR87]{falkrandell:homotopy}
\bysame, 
\emph{On the homotopy theory of arrangements}, 
Complex analytic singularities, 101--124, 
Adv. Stud. Pure Math., 8, North-Holland, Amsterdam, 1987.

\bibitem[Hul16]{hultman:koszul} 
A.~Hultman,
\emph{Supersolvability and the Koszul property of root ideal arrangements},
 Proc. AMS, \textbf{144} (2016), 1401--1413.

\bibitem[LM94]{LeibmanMarkushevich}
A.~Leibman and D.~Markushevich, 
\emph{The monodromy of the Brieskorn bundle}. 
Geometric topology (Haifa, 1992), 91--117, 
Contemp. Math., \textbf{164}, Amer. Math. Soc., Providence, RI, 1994. 

\bibitem[Li06]{LiLi}
Li Li,  
Private communication, 2006.

\bibitem[Mat71]{mather:strat}
J.~Mather,
\emph{Stratifications and mappings} 
in: Dynamical systems (Proc. Sympos., Univ. Bahia, Salvador, 1971) (1973), 195--232.

\bibitem[OT92]{orlikterao:arrangements} P.~Orlik and H.~Terao,
  \emph{Arrangements of hyperplanes}, Springer-Verlag, 1992.

\bibitem[Pa93]{paris:deligne}
L. Paris,
\emph{The Deligne complex of a real arrangement of hyperplanes}, 
Nagoya Math. J. \textbf{131} (1993), 39--65.

\bibitem[R\"o17]{roehrle:ideal}
G.~R\"ohrle,
\emph{Arrangements of ideal type},
J. Algebra, \textbf{484},  (2017), 126--167. 

\bibitem[ST06]{sommerstymoczko}
E.~Sommers and J.~Tymoczko, 
\emph{Exponents for $B$-stable ideals}. 
Trans. Amer. Math. Soc. \textbf{358} (2006), no. 8, 3493--3509. 

\bibitem[Sta72]{stanley:super}
R. P. Stanley, 
\emph{Supersolvable lattices},
Algebra Universalis \textbf{2} (1972), 197--217. 


\bibitem[Ter86]{terao:modular} 
H.~Terao, 
\emph{Modular elements of lattices and topological fibrations},
Adv. in Math. \textbf{62} (1986), no. 2, 135--4. 

\bibitem[Th69]{thom}
R.~Thom,
\emph{ Ensembles et morphismes stratifi\'es}, 
Bull. Amer. Math. Soc. \textbf{75} (1969), 240--284.
\end{thebibliography}

\newcommand{\etalchar}[1]{$^{#1}$}
\providecommand{\bysame}{\leavevmode\hbox to3em{\hrulefill}\thinspace}
\providecommand{\MR}{\relax\ifhmode\unskip\space\fi MR }
\providecommand{\MRhref}[2]{%
  \href{http://www.ams.org/mathscinet-getitem?mr=#1}{#2} }
\providecommand{\href}[2]{#2}


\end{document}